\documentclass[11pt]{amsart}

\usepackage{amssymb}
\usepackage{enumitem}
\usepackage{booktabs, multirow}
\usepackage{array}
\usepackage{url}
\usepackage{mathtools, hyperref,doi}
\usepackage{float}

\newcommand{\PreserveBackslash}[1]{\let\temp=\\#1\let\\=\temp}
\newcolumntype{C}[1]{>{\PreserveBackslash\centering}p{#1}}

\newtheorem{thm}{Theorem}[section]
\newtheorem{cor}[thm]{Corollary}
\newtheorem{lem}[thm]{Lemma}
\newtheorem{prop}[thm]{Proposition}

\newtheorem{exam}{Example}
\newtheorem{rem}[thm]{Remark}

\newcommand{\C}{\mathbb{C}}
\newcommand{\D}{\mathbb{D}}

\newcommand{\norm}[1]{\left\Vert#1\right\Vert}
\newcommand{\abs}[1]{\left\vert#1\right\vert}
\newcommand*\CLOSED[1]{\overline{#1}}

\newcommand*\COP{\mathcal K}  
\newcommand*\BOP{\mathcal L}  

\newcommand*\MV{\mathcal MV} 
\newcommand*\MS{\mathcal MS} 
\newcommand*\MINML{\mathcal ML_{\min}} 
\newcommand*\WC{\mathcal WC} 

\newcommand*\MAXML{\mathcal ML_{\max}} 
\newcommand*\NA{\mathcal{NA}} 

\let\originalleft\left
\let\originalright\right
\renewcommand{\left}{\mathopen{}\mathclose\bgroup\originalleft}
\renewcommand{\right}{\aftergroup\egroup\originalright}

\allowdisplaybreaks

\title[Weak cluster points of maximizing sequences on Banach spaces satisfying $\boldsymbol{(M_p)}$]{Weak cluster points of maximizing sequences on Banach spaces satisfying $\boldsymbol{(M_p)}$}

\begin{document}

\author[ ]{David Norrbo}

\email{d.norrbo@reading.ac.uk}
\address{Department of Mathematics and Statistics, School of Mathematical and Physical Sciences, University of Reading, Whiteknights, PO Box 220, Reading RG6 6AX, UK.}




\begin{abstract}
Given a bounded linear operator $T$ on a separable Banach space with property $(M_p)$, we prove that the smallest and the largest norm of weak cluster points of all maximizing sequences for $T$ can only take the values $0$ or $1$. The three classes of bounded linear operators emerging from the dichotomy of these extremal norm values coincides with the partition, created by considering the norm-attaining property and if the essential norm equals the norm.
\end{abstract}

\maketitle

\section{Introduction}

There are many partitions of the set of bounded operators on Banach spaces, for example, compact and non-compact operators, and norm-attaining and non-norm-attaining operators. In this article, we consider operators on Banach spaces that in some sense are similar to \(\ell^p, \, p>1\) spaces, and what weak cluster points of maximizing sequences tells us about the operator. It turns out that the infimum of the norm of all weak cluster points to a given operator is either $0$ or $1$. This dichotomy also appears when the supremum is considered instead of the infimum. We, therefore, obtain a partition consisting of three classes of bounded operators, which coincides with the partition obtained from considering norm-attaining operators and operators for which the essential norm coincides with the norm.

The interest of characterizing when the essential norm equals the norm arises not only from Banach space geometry, but also from a conjecture concerning the value of the norm of the Hilbert matrix operators on certain weighted Bergman spaces on the complex unit disk, \(A^p_\alpha, \, p>2+\alpha>1\) \cite[p.~516]{Karapetrovic-2018}. The essential norm has been determined in \cite[p.~23]{Lindstrom-2022} and \cite[Theorem 1.2]{Norrbo-2025}; incremental progress on determining the norm can be found in the introduction in \cite{Norrbo-2025}. Finding a suitable sufficient condition for the norm and essential norm to coincide would be a unconventional way of solving the yet unsolved conjecture. 

The Banach algebra of bounded operators and the ideal of compact operators on a Banach space \(X\) are denoted \(\BOP(X)\) and \(\COP(X)\) respectively. The essential norm of an operator \(T\in \BOP(X)\) is defined as \(\norm{T}_e = \inf_{K\in\COP(X)} \norm{T-K} \). Recall that compact operators are always completely continuous, meaning that they map weakly convergent sequences into norm convergent sequences.

For a given \(\lambda\geq 1\), a Banach space is said to have the \(\lambda\)-compact approximation property if for every \(\epsilon>0\) and compact set \(S\), there is a compact operator \(K\) with \(\norm{K}\leq \lambda\) such that \(\sup_{x\in S}\norm{(I-K)x}<\epsilon\). If we don't demand a uniform bound on the norm of the compact operators, we say the space has the compact approximation property, and if \(\lambda =1\), it has the metric compact approximation property. For a separable reflexive space, all of these approximation properties coincide, see \cite[Proposition 1]{Cho-1985} (and also \cite[Remark 1.4]{Godefroy-1988}).

Let \(X\) be a Banach space. A closed subspace \(Y\subset X\) is said to be an \(M\)-ideal if there is a space  \(Z\subset X^*\) such that \(X^*=Z\oplus_1 Y^{\perp}\). A great monograph on M-ideals is \cite{MidealBook}. For a given \(p>1\), a separable Banach space \(X\) is said to have property \((M_p)\) if \(\COP(X\oplus_p X) \) is an \(M\)-ideal in \(\BOP(X\oplus_p X)\). Two useful characterizations of property \((M_p)\) can be found in \cite[Corollary 3.6]{Kalton-1995}. Similarly, \cite[Theorem 2.13]{Kalton-1995} gives a comparable characterization for the weaker notion of \(\COP(X)\) being an \(M\)-ideal in \(\BOP(X)\). A useful property, also weaker than \((M_p)\), is \((m_p)\): For every \(x_n\to 0\) weakly, it holds that 
\begin{equation}\tag{$m_p$} 
\limsup_n \norm{x_n+x}^p = \limsup_n \norm{x_n}^p +\norm{x}^p.  
\end{equation}
 It is worth noting that a space having \((M_p), \, p>1\) is reflexive, see \cite[VI. Proposition 5.2]{MidealBook}. As a consequence, \cite[Corollary 3.6]{Kalton-1995} implies a separable Banach space has \((M_p)\) if and only if it is reflexive, has \((m_p)\) and the compact approximation property. 

We proceed with some further notation. Let  \(T\in\BOP(X)\). The set of maximizing vectors, \(\MV := \MV(T)\), is the set of vectors in \(B_X\) satisfying \(\norm{T x} = \norm{T}\). This set is closed, and by definition non-empty if \(T \) is norm-attaining, which is denoted \(T\in\NA\). Whether or not an operator is norm attaining has been examined in, for example, \cite{Pellegrino-2009,Dantas-2023}. There is also an open problem concerning the denseness of the norm-attaining operators, initiated in \cite{Bishop-1961} (see also \cite{Lindenstrauss-1963}, and \cite{Martin-2016} for further information). It is easy to see that on a reflexive space, the compact operators are norm-attaining, but all norm-attaining operators are not compact.

 Let \(\MS :=\MS(T)\) be the set of maximizing sequences, that is, \(\MS := \{(x_n)\subset \partial B_X : \lim_n\norm{T x_n} = \norm{T}\}\). The set of weak cluster points for sequences in \(\MS\) is given by 
\[
\WC := \WC(T) :=\bigcup_{(x_n)\in \MS}\bigcap_k \CLOSED{(x_n)_{n\geq k}}^w,
\]
which is nonempty if the spaces is reflexive. For a separable reflexive Banach space, we proceed to define the minimal and maximal weak maximizing limit to be
\[
\MINML := \MINML( T)  := \inf \{\norm{x} : x\in\WC(T) \}
\]
and
\[
\MAXML := \MAXML( T)  := \sup \{\norm{x} : x\in\WC(T) \}, 
\]
respectively. Notice that \(  \MINML  ,  \MAXML \in [0,1]  \), and by a diagonalization argument the infimum and supremum are attained. Finally, we denote the unit norm weakly null sequences by \(W_0\). We proceed by presenting the main result, Theorem \ref{thm:main}, followed by some examples in Section \ref{sec:mainSec}. Section \ref{sec:proofs} contains the proof of the main result, and some alternative proofs of some parts of the main result. We also provide a short proof of the essential norm formula stated in \cite{Werner-1992} using Theorem \ref{thm:main}.

\section{The main result and examples}\label{sec:mainSec}


The following theorem also states that if \( \WC\setminus (\{0\}\cup\partial B_X)\) is not empty, then \(\MINML=0\) and \(\MAXML=1\), creating the dichotomy \(\MINML,\MAXML\in\{0,1\}\).

\begin{thm}\label{thm:main}

Let \(X\) be a separable Banach space with \((M_p)\) and let \(T\in \BOP(X)\).  It holds that
\begin{itemize}
\item \(\MINML<1\) if and only if \(\MINML = 0\) if and only if \(\norm{T} = \norm{T}_e\), 
\item \(\MAXML>0\) if and only if \(\MAXML = 1\) if and only if \(T\in \NA\).
\end{itemize}

The result is illustrated in Table \ref{tab:APartitionOfBoundedOperators}:

\begin{table}[H]
\caption{Every row is a partition of \(\BOP(X)\)  and every column corresponds to equivalent statements.}
\begin{tabular}{ p{3cm}  p{3cm}  p{3cm} }
\hline
\multicolumn{3}{c}{\(T\in\BOP(X)\)}\\
\hline
 \(T\not\in \NA\) & \multicolumn{2}{|c}{\(T\in\NA\)}  \\
\hline
 \multirow{ 2}{*}{ \(\MAXML(T)=0\) } & \multicolumn{2}{|c}{\(\MAXML(T)>0\)}   \\ 
\cline{2-3}
 & \multicolumn{2}{|c}{\(\MAXML(T)=1\)}   \\ 
\hline
 \multicolumn{2}{c|}{\(\norm{T}= \norm{T}_e\)} & \(\norm{T}> \norm{T}_e\) \\
\hline
\multicolumn{2}{c|}{\(\MINML(T)<1\)} & \multirow{ 2}{*}{ \(\MINML(T)=1\) } \\
\cline{1-2}
\multicolumn{2}{c|}{\(\MINML(T)=0\)} &  \\
\hline
& & 
\end{tabular}\label{tab:APartitionOfBoundedOperators}
\end{table}

\end{thm}

\begin{rem}\label{rem:ThreeClasses}
From Theorem \ref{thm:main}, it follows that \(\BOP(X)\) can be partitioned into three classes: 
\begin{itemize}
\item \(\MINML=\MAXML=0\) if and only if \(\norm{T} = \norm{T}_e\) and \(T\notin\NA\),
\item \(\MINML = 0\) and \(\MAXML = 1\) if and only if \(\norm{T} = \norm{T}_e\) and \(T\in \NA\), and 
\item \(\MINML=\MAXML=1\) if and only if \(\norm{T} > \norm{T}_e\) and \(T\in \NA\). 
\end{itemize}
The implication \(T\notin \NA \ \Rightarrow \ \norm{T} = \norm{T}_e\) also follows from \cite[Proposition 8 (i)]{Werner-1992}.
\end{rem}

\begin{rem}
Concerning the assumptions in Theorem \ref{thm:main}, it is sufficient to assume \(X\) is a separable reflexive Banach space with \((m_p)\), except for the implication: If  \(\norm{T} = \norm{T}_e \), then there is a weakly null maximizing sequence, see the proof of Proposition \ref{prop:EqualNormImplyMINMLnull}. For this implication we need \(X\) to have the compact approximation property.   
\end{rem}

\begin{rem}\label{rem:WMP}
It is clear that 
\[T\in \NA \ \Longrightarrow \ \MAXML = 1 \ \Longrightarrow \ \MAXML >0.
\]
The implication \(\MAXML >0 \Rightarrow  T\in \NA \) for all \(T\in \BOP(X)\) is the definition of \(X\) having the weak maximizing property, which has been proved for the spaces \(\ell^p, \, p>1\) in \cite{Pellegrino-2009}. The proof for reflexive spaces with \((m_p)\) follows from the calculations done on the 11 last lines of the proof of Theorem 1 in \cite{Pellegrino-2009}. Moreover, this proof shows that if \((x_n)\in \MS\) and \(x_n\to x_0\) weakly for some \(x_0\in B_X\setminus \{0\}\), then \(x_0/\norm{x_0} \in \MV\), that is, \(x_0/\norm{x_0}\) is a maximizing vector for \(T\).

\end{rem}

\begin{exam}
In \cite[Corollary 2.2]{Hennefeld-1973} it is proved that \( \COP(\ell^p) \) is an \(M\)-ideal in \( \BOP(\ell^p) \), for \( p>1\). Since \(\ell^p \oplus_p \ell^p\) is isometrically isomorphic to \(\ell^p\), it follows that \(\ell^p, \, p>1\) satisfies \((M_p)\).

The weighted Bergman spaces \(A^p_v, \, p\geq 1\) consists of complex valued analytic functions on the open, complex, unit disk, \(f\colon \D \to \C\), and are defined as
\[
A^p_v := \bigg\{ f\in \mathcal H(\D) : \norm{f} := \bigg( \int _{\D}  \abs{f(z)}^p v(z) \, dA(z) \bigg)^{\frac{1}{p}}<\infty \bigg\},
\]
where \(dA(z)\) is the normalized uniform area measure \(dA(x+iy) = dx \, dy/\pi\) and \(L^1(\D,dA)\ni v\colon \D \to ]0,\infty[\) is a suitable (weight) function, for example, continuous, radial (\(v(z) = v(\abs{z})\)) and monotone function. Using \cite[Corollary 3.6]{Kalton-1995}, one can prove that the reflexive weighted Bergman spaces \(A^p_v, \, p>1\) has \((M_p)\), see also \cite[Corollary 4.8]{Kalton-1995}. The fact that \(A^p_v\) has \((m_p)\) can be proved similarly to \cite[Proof of Theorem 2]{Brezis-1983}; some details are provided in the corrigendum to \cite{Lindstrom-2022}.
\end{exam}

Next, we give examples of operators in each of the three classes presented in Remark \ref{rem:ThreeClasses}.
\begin{exam}
Consider the weighted Bergman spaces \(A^p_\alpha, \, p>1,\alpha >-1\), where the weight function is given by \(v(z) = (1+\alpha)(1-\abs{z}^2)^\alpha\). For general weights, the only additional assumption we impose in this example, other than a weight yielding the \((M_p)\) property, is that the evaluation maps \(\delta_z\colon A^p_v \to \mathbb C \colon f\mapsto f(z), \, z\in \D\) are bounded. A good monograph on the spaces \(A^p_\alpha\) is \cite{Zhu-2007}. In \cite[Theorem 4.14]{Zhu-2007}, it is proved that the evaluation maps are bounded on \(A^p_\alpha\). 

Consider the multiplication operators \(M_g \colon f \mapsto gf\) on \(A^p_\alpha\), where \(g\colon \D \to \C\) is bounded and analytic. If \(g\) is constant, then \(M_g = c I\) for some \(c\in \C\). Clearly \(M_g\) attains its norm at every unit vector in \(A^p_\alpha\) proving \(\MAXML(M_g) = 1\). For \(n=0,1,2\ldots\), let \(e_n\) denote the normalized version of the basis vector \(z^n\). Then \((e_n)\) is a weakly null sequence, proving \(\MINML(M_g) = 0\). Now, assume that \(g\) is not constant. It is well known that \(\norm{M_g} = \sup_{z\in \D} \abs{g(z)}\); the lower bound for the norm can be obtained using the bounded evaluation maps \(\delta_z\). Moreover, partitioning the open unit disk into a closed disk with radius \(1/2\), \(\overline{(1/2)\D}\),  and its complement, it follows from the maximum modulus principle and the fact that any analytic function has at most finitely many zeros in \(\overline{(1/2)\D}\) that \(\norm{M_g f} < \sup_{z\in \D} \abs{g(z)} \) for all \(f\in B_{A^p_\alpha}\). Hence, \(M_g\)  is not norm-attaining. 

To summarize, for any bounded analytic function \(g\colon \D \to \C\), we have \(M_g\in\BOP(A^p)\) and  \(\MINML(M_g) = 0\). If \(g\) is constant, then \(\MAXML(M_g) = 1\), else \(\MAXML(M_g) = 0\).
\end{exam}

\begin{exam}
The null operator satisfy \(\MINML = 0, \MAXML = 1\), but other compact operators \(K\) cannot have \(\MINML(K) = 0\), therefore, \(\MINML(K) =\MAXML(K)  = 1\). This can be seen, for example, by the fact that they are completely continuous, or the fact that \(\norm{K}>0=\norm{K}_e\) in conjunction with Theorem \ref{thm:main}.
\end{exam}

\begin{exam}
On the complex-valued (or real valued) \(\ell^p(\mathbb N)\)-space, the pointwise multiplication operator can belong to all three classes, depending on the symbol \(b\in \ell^\infty\). The norm is given by \(\norm{M_b} = \sup_n \abs{b(n)}\) and the essential norm is given by \(\norm{M_b}_e = \limsup_n \abs{b(n)}\). We can see that the maximum modulus principle for analytic functions is the reason \(\norm{M_g}_e < \norm{M_g}\) never happens on Bergman spaces. It is, however, easy to see that both \(\norm{M_b}_e < \norm{M_b}\) and \(\norm{M_b}_e  = \norm{M_b}\) can happen on \(\ell^p\). Moreover, if \(\abs{b(k)}<\sup_n \abs{b(n)}\) for all \(k\), then \(\MAXML = 0\).

\end{exam}

\section{Proof of Theorem \ref{thm:main} and some alternative proofs}\label{sec:proofs}


\begin{proof}[Proof of Theorem \ref{thm:main}]
The proof of the equivalences involving \(\MAXML\) and \(\NA\) are contained in Remark \ref{rem:WMP}. Moreover, it is clear that \( \MINML = 0 \) implies both \(\norm{T} = \norm{T}_e \) and \( \MINML <1 \). In \cite{Werner-1992} on page 499, after the proof of Lemma 5, it is claimed (without proof) that 
\begin{equation}\label{eq:essNormFormula}
\norm{T}_e = \sup_{(x_n)\in W_0} \limsup_n \norm{T x_n }
\end{equation}
holds if \(X\) has \((M_p)\). It follows that \(\norm{T} = \norm{T}_e  \ \Rightarrow \ \MINML = 0 \). Finally, assume \(T\in\BOP(X)\) is an operator with \(\MINML <1 \). If \(\norm{T} = \norm{T}_e\), then \(\MINML = 0 \) by \eqref{eq:essNormFormula}. If \(\norm{T} > \norm{T}_e\), the same conclusion, \(\MINML = 0 \), follows from \cite[Lemma 3.4]{Siju-2024}.
\end{proof}

In the proof above, the two statements \(\MINML<1\) implies \(\MINML = 0\), and \(\norm{T} = \norm{T}_e\) implies  \(\MINML = 0 \) relied on the essential norm formula \eqref{eq:essNormFormula}. We proceed with proving these statements without the use of \eqref{eq:essNormFormula}, which in turn would prove Theorem \ref{thm:main} without the use of  \eqref{eq:essNormFormula}. We end the section with a simple proof of the formula for the essential norm \eqref{eq:essNormFormula} using Theorem \ref{thm:main}.

\begin{prop}\label{prop:main}
Let \(X\) be a separable reflexive Banach space with \((m_p)\), and let \(T\in \BOP(X)\). Then \(\MINML<1\) implies \(\MINML = 0\). 
\end{prop}

Proposition \ref{prop:main} follows immediately from the following lemma, whose proof share some similarities with \cite[Proof of Lemma 3.4]{Siju-2024}.

\begin{lem}\label{lem:ImportantEqualityForDictonomy}
Let \( X\) be a separable reflexive Banach space that satisfies property \((m_p)\). Let \(T\in \BOP(X)\) and \(x_0\in \WC(T)\). There exists a sequence \((z_n)\in W_0\) such that
\[
\norm{T}^p (1 -  \norm{x_0}^p) =  \lim_n \norm{Tz_n}^p  (1-\norm{x_0}^p ). 
\]
\end{lem}
\begin{proof}

Let \((y_n)\in \MS\) with \(y_n\to x_0\) weakly. By going to a subsequence if necessary, we can assume that \( \lim_n \norm{ y_n - x_0 }^p\) exists, and that either \( \lim_n \norm{y_n - x_0} = 0\)  or \( \inf_n\norm{y_n - x_0}>0\). If \( \lim_n  \norm{y_n - x_0} = 0\),  we have \(\norm{x_0}=1\) and we are done. If  \(\inf_n\norm{y_n - x_0}>0\), put \(z_n = (y_n-x_0)/\norm{ y_n-x_0}\). Property \((m_p)\) yields  \( 1 - \norm{x_0 }^p = \limsup_n \norm{ y_n - x_0 }^p  \). Similarly, we have \( \norm{ T }^p - \norm{ Tx_0 }^p = \limsup_n \norm{ T(y_n - x_0) }^p  \). It follows that \(\norm{ T }^p - \norm{ Tx_0 }^p =  \limsup_n \norm{ Tz_n }^p (1-\norm{ x_0 }^p) \). Combining this with the elementary inequalities
\[
\norm{ T }^p - \norm{ Tx_0 }^p   \geq  \norm{ T }^p (1-\norm{x_0}^p)  \geq \limsup_n \norm{ Tz_n }^p (1-\norm{ x_0 }^p)
\]
 yields the statement.

\end{proof}

Property \((m_p)\) yields that \(\MINML<1\) is equivalent to having a maximizing sequence with no norm-convergent subsequences. Utilizing the fact that \(\MINML = 0 \ \Rightarrow \ \norm{T} = \norm{T}_e\), Proposition \ref{prop:main} can be expressed as follows:

\begin{cor}\label{cor:main}
Let \(X\) be a separable reflexive Banach space with \((m_p)\), and let \(T\in \BOP(X)\). Then \(\norm{T} = \norm{T}_e\) if there is a maximizing sequence for \(T\) with no norm-convergent subsequences.
\end{cor}


\begin{prop}\label{prop:EqualNormImplyMINMLnull}
Let \(X\) be a separable Banach space with \((M_p)\) and let \(T\in \BOP(X)\). It holds that \(\norm{T} = \norm{T}_e  \ \Rightarrow \ \MINML = 0 \).
\end{prop}

\begin{proof}
We use contraposition to prove this statement. It is sufficient to assume \(T\not\equiv 0\) and \(\MINML>0\). By Proposition \ref{prop:main}, we may assume \(\MINML = 1\). In order to prove that \(\norm{T} > \norm{T}_e\), let \((K_n)\subset B_{\COP(X)}\) be a sequence satisfying \( \lim_n \norm{I-K_n} = 1\) and \( \lim_n \norm{(I - K_n) x} = 0\) for every \(x\in X\); this is possible due to \cite[Theorem 2.4 (4)]{Kalton-1993} (see also \cite[Lemma 5.1]{Harmand-1984}). Moreover, property \((m_p)\) implies that any sequence \((x_n)\subset B_X\) that converges weakly to an \(x_0\in \partial B_X\) must satisfy \( \lim_n \norm{x_n-x_0} = 0\). Hence, since the space is reflexive and \(\MINML = 1\), every \((x_n)\in \MS\) will have a norm convergent subsequence. It follows that \(MV\) is compact and
\begin{equation}\label{eq:MScapMVnotEmpty}
\CLOSED{\{x_n\}}^{\norm{\cdot}} \cap \MV \neq \emptyset  \quad  \text{ for every }  (x_n)\in \MS.
\end{equation}

Since \(\MV\) is compact, we have

\[
\lim_n\sup_{x\in\MV}  \norm{(I-K_n)Tx} = 0.
\]
It follows that there is an open set \(N\supset\MV\) that is independent of \(n\) such that
\begin{equation}\label{eq:nearMV}
\limsup_n \sup_{x\in N}  \norm{(I-K_n)Tx} <\norm{T}.
\end{equation}
Furthermore, the closed set \(B_X\setminus N\) is disjoint from \(\MV\), therefore, \eqref{eq:MScapMVnotEmpty} yields
\[
(B_X\setminus N)^{\mathbb N} \cap \MS = \emptyset.
\]
We can now conclude that

\[
\limsup_n \sup_{x\in  B_X \setminus N}  \norm{(I-K_n)Tx} \leq \limsup_n \norm{I-K_n}    \sup_{x\in  B_X \setminus N}  \norm{Tx} <\norm{T},
\]
and hence, using \eqref{eq:nearMV},
\[
\norm{T}_e = \inf_K  \max\Big\{  \sup_{x\in N\cap B_X }\norm{(T-K)x}  ,   \sup_{x\in B_X\setminus N }\norm{(T-K)x}   \Big\} < \norm{T}.
\]

\end{proof}

We end the section with a proof of the essential norm fomula for operators on separable Banach spaces with \((M_p)\), stated on \cite[p.~499]{Werner-1992}.

\begin{prop}
Let \(X\) be a separable Banach space with \((M_p)\). The essential norm can be represented as
\[
\norm{T}_e = \sup_{(x_n)\in W_0} \limsup_n \norm{Tx_n}.
\]
\end{prop}
\begin{proof}
Since the space of compact operators \(\COP(X)\) is proximal in \(\BOP(X)\) \cite[II. Proposition 1.1]{MidealBook}, there is a \(K\in \COP(X)\) such that
\[
\norm{T}_e =\norm{T-K}. 
\]
Now, we can apply Theorem \ref{thm:main} to the the operator \(T-K\), to obtain that either \(\MINML(T-K) = 0\), in which case we are done, or \(\norm{T-K}> \norm{T-K}_e\). However, the latter cannot occur since \(K\) is the compact operator minimizing \(\norm{T-K}\), so \(\norm{T-K} = \norm{T-K}_e\) and we are done.  
\end{proof}

\section{Ackowledgements}

The author was financially supported by the Magnus Ehrnrooth Foundation.

\end{document}